\documentclass[11pt]{article}

\usepackage{parskip}
\usepackage[T1]{fontenc}
\usepackage{amsmath}
\usepackage{amsfonts}
\usepackage{amssymb}
\usepackage{amsthm}
\usepackage{fancyhdr}
\usepackage{tikz,tkz-graph}
\usetikzlibrary{hobby}
\usepackage{subcaption}
\usepackage[pass,letterpaper]{geometry}
\usepackage[shortlabels]{enumitem}
\usepackage{graphicx}
\usepackage{subcaption}
\usepackage[textsize=scriptsize, backgroundcolor=orange!10]{todonotes}
\usepackage{thmtools, thm-restate}
\usepackage{wasysym}

\usepackage{etoolbox}

\makeatletter

\patchcmd\deferred@thm@head
  {\addvspace{-\parskip}}
  {}
  {}{\typeout{\string\deferred@thm@head patch failed!}}

\makeatletter 

\newtheorem{main}{Theorem}
\newtheorem*{conjecture}{Conjecture}
\newtheorem*{define}{Definition}
\newtheorem{theorem}{Theorem}[section]
\newtheorem*{theorem*}{Theorem}
\newtheorem{lemma}[theorem]{Lemma}
\newtheorem*{lemma*}{Lemma}
\newtheorem*{question*}{Question}

\newtheorem*{proposition*}{Proposition}
\newtheorem{definition}[theorem]{Definition}

\newtheorem{remark}[theorem]{Remark}

\definecolor{re}{HTML}{E2322E} 
\definecolor{bl}{HTML}{071344} 
\definecolor{gr}{HTML}{60BE5B} 

\def\E{\mathbb{E}}

\def\LL{\mathcal{L}}
\def\M{\mathcal{M}}
\def\HH{\mathcal{H}}

\def\T{\mathcal{T}}

\newcommand{\abs}[1]{\left\vert#1\right\vert}

\begin{document}
\title{On multicolor Ramsey numbers of triple system paths of length 3}

\author{
Tom Bohman\thanks{Department of Mathematical Sciences, Carnegie Mellon University, Pittsburgh, PA 15213, USA. Email: {\tt tbohman@math.cmu.edu}, supported in part by a grant from the Simons Foundation (587088, TB).}
\and 
Emily Zhu\thanks{Department of Mathematics, University of California at San Diego, La Jolla, CA 92093, USA. Email: {\tt e9zhu@ucsd.edu}, supported in part by the National Science Foundation Graduate Research Fellowship Program under Grant No. DGE-2038238.}
}
\date{}

\maketitle

\begin{abstract}
Let $\HH$ be a 3-uniform hypergraph. 
The multicolor Ramsey number $ r_k(\HH)$ is the smallest integer $n$ such that every coloring of $ \binom{[n]}{3}$ with $k$ colors has a monochromatic copy of $\HH$.  
Let $ \LL$ be the loose 3-uniform path with 3 edges and $ \M$ denote the messy 3-uniform path with 3 edges; that is, let $\LL = \{abc, cde, efg\}$ and $\M = \{ abc, bcd, def\}$.  
In this note we prove $ r_k(\LL) < 1.54k$ and $ r_k(\M) < 1.6k$ for $k$ sufficiently large.  
\end{abstract}

\section{Introduction}
Let $r \ge 2$ and consider an $r$-uniform hypergraph $\HH$.  The multicolor Ramsey number $r_k(\HH)$ is the minimum $n$ such that every 
$k$-coloring of $\binom{[n]}{r}$ contains a monochromatic copy of $\HH$.  The problem of determining the asymptotics of $ r_k(\HH)$ is wide
open even for some simple $\HH$.  Consider, for example, the graph triangle $K_3$.  It is known that $ r_k(K_3) $ is at least exponential in $k$ and that the
limit as $k $ tends to infinity of $ r_k(K_3) ^{ 1/k } $ exists.  However, the value of this limit remains an open problem; indeed, it is an old \$250 problem of
Erd\H{o}s to determine this limit and a \$100 problem to just determine whether or not this is finite \cite{chung1998}.

In this work we consider $ r_k ( \HH )$ where $ \HH$ is a 3-edge, 3-uniform path.  There are three such hypergraphs: The tight path $ \T= \{ abc, bcd, cde\}$, the loose path $ \LL = \{abc, cde, efg\}$,
 and the messy path $\M = \{ abc, bcd, def\}$.

\begin{figure}[ht]
\begin{subfigure}{.33\textwidth}
\centering
\begin{tikzpicture}[scale=.7]
\SetVertexSimple
\renewcommand*{\VertexInnerSep}{0pt}
\renewcommand*{\VertexSmallMinSize}{3pt}
\Vertices[dir=\EA]{line}{a,b,c,d,e}
\draw[rounded corners=6,color=re,thick] (1,0) ellipse (1.4 and .3);
\draw[rounded corners=6,color=gr,thick] (2,0) ellipse (1.4 and .3);
\draw[rounded corners=6,color=bl,thick] (3,0) ellipse (1.4 and .3);
\end{tikzpicture}
\caption{Tight Path}
\end{subfigure}\begin{subfigure}{.33\textwidth}
\centering
\begin{tikzpicture}[scale=.7]
\SetVertexSimple
\renewcommand*{\VertexInnerSep}{0pt}
\renewcommand*{\VertexSmallMinSize}{3pt}
\Vertices[dir=\EA]{line}{a,b,c,d,e,f,g}
\draw[rounded corners=6,color=re,thick] (1,0) ellipse (1.4 and .3);
\draw[rounded corners=6,color=gr,thick] (3,0) ellipse (1.4 and .3);
\draw[rounded corners=6,color=bl,thick] (5,0) ellipse (1.4 and .3);
\end{tikzpicture}
\caption{Loose Path}
\end{subfigure}\begin{subfigure}{.33\textwidth}
\centering
\begin{tikzpicture}[scale=.7]
\SetVertexSimple
\renewcommand*{\VertexInnerSep}{0pt}
\renewcommand*{\VertexSmallMinSize}{3pt}
\Vertices[dir=\EA]{line}{a,b,c,d,e,f}
\draw[rounded corners=6,color=re,thick] (1,0) ellipse (1.4 and .3);
\draw[rounded corners=6,color=gr,thick] (2,0) ellipse (1.4 and .3);
\draw[rounded corners=6,color=bl,thick] (4,0) ellipse (1.4 and .3);
\end{tikzpicture}
\caption{Messy Path}
\end{subfigure}
\end{figure}

The tight path was studied in \cite{axen2014}, where it is shown that $2k(1-o(1)) \leq r_k(\T) \leq 2k+3$.  The tight path is somewhat different from $\LL$ and $\M$ as
the tight path has a transversal vertex, i.e., a vertex contained in every edge.  Thus, the problem of determining $ r_k( \T)$ is related to the problem of determining the multicolor Ramsey number of the graph path with 4 vertices, $P_4$. It is known that $2k \leq r_k(P_4) \leq 2k + 2$ \cite{irv1974}. For further results on the multicolor Ramsey numbers of longer graph paths, see \cite{sar2016,davies2017,knierim2019}.

The best known lower bounds on $ r_k( \LL)$ and $ r_k(\M)$ are
\[ r_k( \LL) \ge k+6  \ \ \ \text{ and } \ \ \ r_k(\M) \ge k +5. \]
The constructions that provide these lower bounds have a common structure.  Let $n$ be one less than the bound we are establishing (so $n = k+5$ for the loose path and $n = k+4$ for the messy path).
We begin by ordering the vertex set $ V = \{ v_1, v_2, \dots , v_n \}$.  A triple $ v_x v_y v_z$ with $ x < y < z $ and $ x < k $ is assigned color $ x$.  The remaining triples are assigned color $k$.  The
first $k-1$ colors give stars and therefore do not contain copies of either path.  The final color is assigned to a complete subhypergraph, but the number of vertices is one fewer than the number of vertices 
in the path in question.  It is believed that these lower bounds give the actual multicolor Ramsey numbers for these hypergraphs \cite{polcyn2017r10, luczak2016, luczak2018}.

In this work, we provide improvements on the upper
bounds on these multicolor Ramsey numbers.  The previous best known 
result for the loose path was $ r_k( \LL) < 1.975k + 7 \sqrt{k} $, which was established by \L uczak and Polcyn \cite{luczak2018}.  We are not
aware of any discussion of the multicolor Ramsey number of the messy path in literature.  Our main results are as follows:
\begin{main}
\label{main:loose}
If $k$ is sufficiently large then
\[ r_k( \LL) < 1.531 k.\]
\end{main}
\begin{main}
\label{main:messy}
If $ \varepsilon >0$ and $k$ is sufficiently large then
\[ r_k(\M) < \left(\frac{ 10 + \sqrt{19}}{9} + \varepsilon \right)k < (1.596+\varepsilon)k.\]
\end{main}
The proofs of these Theorems are similar, and each has two parts.  Let $\HH $ be $ \LL$ or $\M$.  The first part of the proof is a structural characterization of $\HH$-free hypergraphs; in particular,
we show that an appropriately chosen core of an $\HH$-free hypergraph has a well-organized structure. Such a characterization was first provided 
by \L uczak and Polcyn for the
loose path. See Lemma~19 in \cite{luczak2017paths}.  In the second part of the proof we consider a $k$-coloring of $ \binom{[n]}{3} $ which does not contain a monochromatic copy of $\HH$.  Based on
the structural characterization, we introduce a digraph on vertex set $ [n]$ in each color.  We then proceed to analyze the structure of this colored collection of 
digraphs to produce the bound on the Ramsey number.  

In the case of the loose path, this structural analysis invokes both the Caccetta-H\"aggkvist Conjecture and the 
Triangle Removal Lemma for digraphs.
\begin{conjecture}[Caccetta-H\"aggkvist \cite{caccetta1978}] If $D$ is a digraph on $n$ vertices with no parallel arcs and minimum in-degree at least $r$
then $D$ has a directed cycle with length at most $ \lceil n/r \rceil $.
\end{conjecture}
\noindent We make use of the $ r = n/3$ case of this Conjecture.  As even this special case is still open, we introduce
the following definition.
\begin{define}
We say that a constant $ \alpha \in [1/3,2/5) $ is directed triangle sufficient if every oriented graph $D$ with minimum in-degree at least $ \alpha |V(D)| $ has a directed cycle of length 3.
\end{define}
Of course, the statement that $ 1/3$ is directed triangle sufficient is a special case of the Caccetta-H\"aggkvist Conjecture.  The
best known result for this special case is that $  0.3465 $ is directed triangle sufficient \cite{hladky2017}.  Thus, Theorem~\ref{main:loose} is a Corollary of the following Theorem.
\begin{main}
\label{main:loose2}
If $\alpha$ is directed triangle sufficient, $ \varepsilon >0$ and $ k$ is sufficiently large then
\[ r_k( \LL) < \frac{ 1 + \varepsilon}{ 1 - \alpha} k. \]
\end{main}

Note that if the special case of the Caccetta-H\"aggkvist holds then $ 1/3$ is directed triangle sufficient and we would
achieve the bound $ r_k( \LL) < (3/2 + \varepsilon)n$.

The organization of this paper is as follows. In the rest of this section we introduce definitions and notation for hypergraphs and graphs.  In Section~2, we study the loose path, providing a self-contained proof for a characterization of loose path-free hypergraphs which we then use to prove Theorem~\ref{main:loose2}.
In Section~3, we study the messy path,  establishing a characterization of messy path-free hypergraphs, and proving Theorem~\ref{main:messy}.  Section 4 gives the exact extremal number for the messy path, a result that may be of independent interest.

\subsection{Definitions and notation}
We adopt the convention of identifying a hypergraph $ \HH$ with the edge set of $ \HH$.  We let $ V= V(\HH)$ denote the 
vertex set of a hypergraph $\HH$.   A hypergraph $\HH$ on vertex set $V$ is $r$-uniform if $\HH \subseteq \binom{V}{r}$ where $\binom{V}{r}$ denotes all subsets of $V$ of size $r$. For convenience, we may denote an edge $\{v_1,\dots,v_r\} \in \HH$ by $v_1v_2\dots v_r$.  All hypergraphs considered in this work are $3$-uniform.

The multicolor Ramsey number for a hypergraph is closely linked to its extremal number. We define $ex^{(r)}(n,\HH)$, the {\em extremal number} of $\HH$, to be the maximum number of edges in any $\HH$-free $r$-uniform hypergraph on $n$ vertices. Analogously, we define $Ex^{(r)}(n,\HH)$ to be the {\em extremal family} of $\HH$.   This is the set of 
$\HH$-free $r$-uniform hypergraphs on $n$ vertices and $ex^{(r)}(n,\HH)$ edges. In this work we consider only the case $r=3$, and so we will simply write $ex(n,\HH)$ and $Ex(n,\HH)$ respectively.

Our analysis uses the following concepts. We define the {\em trace} (sometimes known as a link) of some vertex or set of vertices as 
\[Tr(x_1,\dots,x_k) := \{e \setminus \{x_1,\dots,x_k\} \mid e\in \HH, \{x_1,\dots,x_k\} \subseteq e\}.\] The degree of a vertex or set of vertices is then simply $deg(x_1,\dots,x_k) := \abs{Tr(x_1,\dots,x_k)}$. For a 3-uniform hypergraph, we will often refer to $deg(x,y)$ as the codegree of the pair $x,y$. We define the {\em $m$-core} of a hypergraph to be the subhypergraph formed by iteratively removing vertices of degree less than $m$ until every vertex has degree at least $m$ (or the hypergraph is empty).

We also define notation for subhypergraphs. For a hypergraph $\HH$ and $U \subseteq V(\HH)$, we will denote the subhypergraph induced by $U$ by $\HH[U] := \{e \in \HH: e \subseteq U\}$. 
We extend the definition and notation for induced subhypergraphs to graphs and digraphs in the natural way.  

If $G$ is a graph then the {\em matching number} of $G$, denoted $\nu(G)$, is the maximum number of edges in a matching in $G$. Furthermore, the {\em vertex cover number} of $G$, denoted $\tau(G)$, is the minimum number of vertices in a vertex cover of $G$, i.e., a set of vertices which intersects every edge.


\section{Loose Path}
The loose path has been studied extensively: a series of papers examines the extremal number, variations of the extremal number, small cases of the multicolor Ramsey number, and asymptotics of the multicolor Ramsey number. The unique extremal $\LL$-free hypergraph is the complete star when $n$ is at least 8 \cite{jackowska2016turan}.  This implies $k + 6 \leq r_k(\LL) \leq 3k$ for $k \geq 3$ \cite{jackowska2015r3, jackowska2016turan}. 
It is known that $r_k(\LL) = k+6$ for $k \leq 10$ \cite{gyarfas2012loose, jackowska2015r3, jackowska2015r47, polcyn2017r89, polcyn2017r10}.    

The previous best known asymptotic upper bound was 
$r_k( \LL) \leq \lambda k + 7\sqrt{k}$ where $\lambda \approx 1.975$ is a solution to the equation $( \gamma^3 - 3 \gamma^2 + 6\gamma  - 6)^2 - 72\gamma (2 - \gamma)(\gamma - 1)^2=0$ \cite{luczak2018}.
This result was established by first giving a characterization of loose path-free hypergraphs \cite{luczak2017paths} and then 
applying properties of this characterization to each color
in a loose-path free $k$-coloring of $\binom{[n]}{3}$ in order 
to find an upper bound on $n$. 
We use a similar approach, but arrive at a better bound by encoding more information in a digraph for each color class.

The following Section contains a short, self-contained proof of our characterization of loose path-free hypergraphs. We emphasize that this characterization follows from a similar and stronger result due to \L uczak and Polcyn (namely, Lemma 19~in \cite{luczak2017paths}). We include a short proof of the characterization in the interest of completeness. We then apply this
characterization in Section~2.2 to prove Theorem~\ref{main:loose2}.


\subsection{Loose Path-Free Hypergraph Characterization}

\begin{theorem}
\label{thm:loosechar}
If $\HH$ is a loose path-free hypergraph and $\HH'$ is the 22-core of $\HH$ then $\HH'$ has the following structure. The vertex set $V(\HH')$ has a partition into 3 sets $X,Y,Z$, where the set $X$ is partitioned into sets of size 2 and the set $Z$ is partitioned into sets $(A_v : v \in Y)$.  All triples $e$ of the hypergraph $\HH'$ have one of the following two forms:
\begin{itemize}[nosep]
\item $e \cap X$ is one of the pairs in the partition of $X$ and $\abs{e \cap Y} = 1$
\item $e \cap Y = \{y\}$ and $e \setminus \{y\} \subseteq A_y$
\end{itemize}
\end{theorem}
\begin{figure}[ht]
	\centering
	\begin{tikzpicture}[scale=1.4]
	\draw (0,-.5) ellipse (1 and .5) node {$Y$};
	\draw (1.2,.25) ellipse (.1 and .2)
		  (1.2,-.25) ellipse (.1 and .2)
		  (1.2,-.75) ellipse (.1 and .2);
	\draw (1.4,.45) -- (1.6,.45) -- (1.6,-1.05) -- (1.4,-1.05);
	\node[font=\small] at (1.8,-.25) (p) {$X$};

	\draw[rounded corners,color=re] (1.3,-1.08) -- (1.33,-.52) -- (.25,-.35) -- cycle;
	\draw[rounded corners,color=re] (1.33,.02) -- (1.33,-.52) -- (.5,-.25) -- cycle;
	\filldraw (1.2,-.65) circle (.03)
			  (1.2,-.85) circle (.03)
			  (1.2,-.15) circle (.03)
			  (1.2,-.35) circle (.03)
			  (.47,-.44) circle (.03)
			  (.7,-.25) circle (.03);
						
	\filldraw (.25,-.18) circle (.03)
			  (-.3,-.3) circle (.03)
			  (-.7,-.25) circle (.03);
						
	\draw[rounded corners,color=bl]
				(-1.05,.3) -- (-.66,-.42) -- (-.6,.5)
				(-.5,.6) -- (-.315,-.49) -- (-0,.45)
				(.1,.65) -- (.23,-.4) -- (.5,.5);
	\draw (-1.1,.6) -- (-1.3,.6) --(-1.3,-.2) -- (-1.1,-.2);
	\node[font=\small] at (-1.5,.2) {$Z$};
				
	\node[font=\footnotesize] at (-.76,.1) (1) {$A_1$};
	\node[font=\footnotesize] at (-.27,.2) (2) {$A_2$};
	\node[font=\footnotesize] at (.28,.25) (3) {$A_3$};
	\end{tikzpicture}
	\caption{$\LL$-free hypergraph}
\end{figure}

When we apply this Theorem below we make use of the following definitions.
\begin{definition}
\label{def:lock}
Each pair in the partition of $X$ is called a locked pair.
\end{definition}
\begin{definition}
\label{eq:stray}
Let $ \HH$ be an $\LL$-free hypergraph and let $\HH'$ be the 22-core of $\HH$.  We call the triples in $ \HH \setminus \HH'$ 
stray triples (or removed triples).
\end{definition}

The remainder of this subsection is a proof of Theorem~\ref{thm:loosechar}.  Throughout the proof we let $\HH$ be a loose path-free hypergraph and $\HH'$ be the 22-core of $\HH$.

\begin{lemma}
\label{lem:no23match}
The matching number $\nu(Tr_{\HH'}(v)) \neq 2,3$ for all vertices $v \in V(\HH')$.
\end{lemma}
\begin{proof}
Let $v$ be a fixed vertex and assume for the sake of contradiction that the matching number of the trace of $v$ in $\HH'$ is 2 or 3. Let $M$ be a maximal matching in $Tr_{\HH'}(v)$ and let $\mathcal{M}$ be the union of the edges in $M$. 
As the vertex set $\mathcal{M}$ contains at most 15 edges, there is an edge $e$ of $Tr_{\HH'}(v)$ that is not contained in $\mathcal{M}$. As $M$ is a maximal matching, $e$ intersects $\mathcal{M}$ in one vertex. Let $u$ be the vertex in $e$ that is not in $\mathcal{M}$. 
As any edge of $Tr_{\HH'}(v)$ that contains $u$ must also intersect $\mathcal{M}$, there are at most 6 such edges (this count includes $e$ itself). As $u$ is in at least $22 >  \binom{6}{2} + 6 $ edges in $\HH'$, there is a triple $uyz \in \HH'$ such that at least one of $y$ and $z$ is not in $\mathcal{M}$ and neither $y$ nor $z$ is $v$. Consider such a triple $uyz \in \HH'$. A loose path appears among two edges of $M\cup\{e\}$ (expanded to triples by including $v$) and $uyz$. 
\end{proof}

We are now ready to identify the first part of the structure defined in Theorem~\ref{thm:loosechar}. We define $Y$ to be the set of vertices $y$ such that the matching number of $Tr_{\HH'}(y)$ is at least 4.
\begin{lemma}
\label{lem:loosecontain}
If $y \in Y$ and a triple $abc \in \HH'$ intersects an edge $e$ of $Tr_{\HH'}(y)$ then either $e \subset abc$ or $y \in abc$.
\end{lemma}
\begin{proof}
Suppose $y \notin abc$ and $\abs{e \cap abc} = 1$. Let $M$ be a maximum
matching of $Tr_{\HH'}(y)$.  As this matching has at least 4 edges, the edge set $M \cup \{e\}$ contains an edge that intersects $abc$ in 1 vertex and an edge that does not intersect $abc$ and is disjoint from the first edge.  These two edges (expanded to triples by including $y$) and $abc$ form a loose path.
\end{proof}
\noindent We now note that Lemma~\ref{lem:loosecontain} implies that 
no triple in $ \HH'$ contains more than one vertex of $Y$. Indeed, 
assume for the sake of contradiction that $y_1y_2a \in \HH'$ where $y_1,y_2 \in Y$.  Consider a triple that contains $y_1$.  Such a triple intersects the edge $ y_1a$ of $ Tr_{\HH'}(y_2)$. It follows that either the triple contains $y_2$ or the triple contains $a$.  This then implies that the 
matching number of $ Tr_{\HH'}(y_1)$ is at most 2, which is a contradiction.

\begin{lemma}
\label{lem:largecontain}
Let $y \in Y$. Every connected component of $Tr_{\HH'}(y)$ has either at most 2 vertices or at least 23 vertices. Furthermore, if $C$ is the vertex set of a component with at least 23 vertices and $abc \in \HH'$ intersects $C$ then $y \in abc$.
\end{lemma}
\begin{proof}
Consider a connected component with at least 3 vertices. Let $u$ be a vertex of this component of maximum degree; note that $deg(u) \geq 2$. The triples of $\HH'$ that contain $u$ either contain $y$ and therefore correspond to edges of $Tr_{\HH'}(y)$ or contain all of the neighbors of $u$ in $Tr_{\HH'}(y)$ (by Lemma~\ref{lem:loosecontain}). 
The latter condition cannot be satisfied if the degree of $u$ in $Tr_{\HH'}(y)$ is greater than 2. On the other hand, if the degree of $u$ in $Tr_{\HH'}(y)$ is 2 then we have at most 3 triples of $\HH'$ that contain $u$, which is a contradiction. We conclude that the degree of $u$ in $Tr_{\HH'}(y)$ is its degree in $\HH'$, which is at least 22, and so there are at least 23 vertices in the component.

To prove the second assertion in the Lemma, assume for the sake of
contradiction that the triple 
$abc \in \HH'$ contains a vertex $u$ of a component $C$ that has at least 23 vertices but does not contain $y$. Then, by repeated application of Lemma~\ref{lem:loosecontain}, we see that $abc$ contains all vertices of $C$. As $3<23$ this is a contradiction.
\end{proof}
\noindent
 We say that for a vertex $y \in Y$, a connected component of $Tr_{\HH'}(y)$ is large if it has at least 23 vertices. 

We are now ready to identify the other parts of the vertex partition set forth in Theorem~\ref{thm:loosechar}.
For each vertex $y \in Y$, let $A_y$ be the union of the vertex sets of the large components of $Tr_{\HH'}(y)$. Set $Z := \bigsqcup_{y \in Y} A_y$. 
As no triple in $\HH'$ contains more than one vertex of $Y$,
the sets $Y$ and $Z$ are disjoint. Set $X = V(\HH') \setminus (Y \cup Z)$.

We have our partition $X,Y,Z$ and the partition of $Z$ into $(A_v: v \in Y)$. We now look to partition $X$. Let $x \in X$.
It follows from Lemma~\ref{lem:no23match} and the definition of $Y$ that $Tr_{\HH'}(x)$ is a star with at least 22 edges. Let $x'$ be the center of this star. Note that $x' \notin Y$: if $x' \in Y$ then $x$ itself would be in a large component of $A_{x'}$ and so $x$ would be in $Z$. 
Furthermore, $x' \notin Z$ as this would imply - by Lemma~\ref{lem:largecontain} - that all at least 22 triples containing $x$ must contain both $x'$ and some fixed element of $Y$ (of course, there is only one such triple). It follows that $x' \in X$.  We conclude that $X$ can be partitioned into a collection of pairs $xx'$ with the property that every triple of $\HH'$ that contains one vertex in such a 
pair also contains the other vertex in the pair. Finally, note that the third vertex in such a triple must be in the set $Y$. Thus, triples intersecting $X$ are as stated.



\subsection{Multicolor Ramsey number for the loose path: Proof of Theorem~\ref{main:loose2}}

In this Section we prove Theorem~\ref{main:loose2}.  The proof invokes  
the Caccetta-H\"aggkvist Conjecture and the Removal Lemma for directed triangles.  We recall the following for reference.
\begin{definition}
We say that a constant $ \alpha \in [1/3,2/5) $ is directed triangle sufficient if every oriented graph $D$ with minimum in-degree at least $ \alpha |V(D)| $ has a directed cycle of length 3.
\end{definition}
\begin{theorem}[Alon, Shapira \cite{alon2004}, Digraph Removal]
For every fixed $\delta, h$, there is a positive constant $c(h,\delta)$ with the following property. If $H$ is a fixed digraph on $h$ vertices and $G$ is a digraph on $n$ vertices, where $n$ is sufficiently large, with the property that upon the removal of at most $\delta n^2$ arcs $G$ still contains a copy of $H$ then $G$ contains at least $c(h,\delta)n^h$ copies of $H$.
\end{theorem}

Let $ \alpha$ be directed triangle sufficient and let $\varepsilon > 0 $ be a small constant.  Suppose $$ n = \frac{ 1 + \varepsilon }{ 1 - \alpha } k , $$
and assume for the sake of contradiction that there is a loose path-free $k$-coloring of $\binom{[n]}{3}$. Fix such a coloring and let $C$ be the set of colors.  For each color $c \in C$, let $ \HH_c $ be the collection of triples colored with color $c$.  We apply Theorem~\ref{thm:loosechar}.  For each $ \HH_c$, we let $ \HH'_c $ be the 22-core of $ \HH_c$ and we let $ X_c, Y_c, Z_c$ be the
partition of the vertex set given by Theorem~\ref{thm:loosechar}.  Furthermore, for   
each vertex $v$ in the set $Y_c$ let $A_{v,c}$ be the set $A_{v}$ given by Theorem~\ref{thm:loosechar}.  
We define a colored multidigraph $M$ on the vertex set $V = [n]$ as follows. The directed arc $(u,v)$ appears in the multidigraph with color $c$ if $u \in A_{v,c}$.
For a specific color, we denote this arc by $(u,v,c)$. We will also include both $(u,v,c), (v,u,c)$ in the multidigraph if $u,v$ is a locked pair in color $c$ (recall Definition~\ref{def:lock}). 
Our main focus in the proof will be on the pairs of vertices that have arcs of $M$ going in only one direction; this will be most pairs. 

We define the in and out-degrees of the colored multidigraph $M$ as follows:
\[m^-_c(v) = \abs{\{u \in [n]: (u,v,c) \in M\}}, \quad m^-(v) = \sum_{c \in C} m^-_c(v)\]
\[m^+_c(v) = \abs{\{u \in [n]: (v,u,c) \in M\}}, \quad m^+(v) = \sum_{c \in C} m^+_c(v)\]
Note that $m^+(v) \leq k$ as $v$ can appear either in $ A_{y,c}$ for at most one vertex $y$ or in a single locked pair in color $c$.

\begin{lemma}
\label{lem:uncover}
At most $O(k)$ pairs of vertices $\{u,v\}$ have the property that neither $(u,v)$ nor $(v,u)$ appears as an arc in any color in $M$.
\end{lemma}
\begin{proof}
Note that if neither $(u,v,c)$ nor $(v,u,c)$ appears then at most one triple containing $\{u,v\}$ appears in $ \HH'_c$.  So, if neither $ (u,v)$ nor $(u,v)$ appears in
any color then at least $ n-k-2$ triples containing $u$ and $v$ are stray triples.  
As there are at most $ 21nk$ stray triples across the $k$ colors, we see that the number
of pairs $\{u,v\}$ that span no arc of $M$ is at most $\frac{21nk}{n-k-2} = O(k)$.
\end{proof}

We will refer to pairs $\{u,v\}$ such that neither $(u,v)$ nor $ (v,u)$ appears in $M$ as {\em uncovered pairs}.

We define an oriented graph $D$ on $[n]$ as follows
\begin{itemize}[nosep]
\item $(v,u) \in D$ if $(v,u) \in M \text{ and } (u,v) \notin M$
\item $d^+(v)= \left| \{u \in [n]: (v,u) \in D\} \right|$
\item $d^-(v)= \left| \{u \in [n]: (u,v) \in D\} \right|$
\end{itemize}
Note that almost all vertices $v$ have $ d^-(v) > n - k - o(k) $: the out-degrees in $M$ are at most $k$ and there are $ O(k)$ uncovered pairs while the remaining arcs at a vertex are then in-arcs in $D$.

We now apply Directed Triangle Removal.
Set
\[ \delta = \varepsilon^2/17 \]
and consider any digraph $ D'$ formed by deleting $ \delta n^2$ arcs of $D$.  We claim that $D'$ has a directed triangle. Let
$X$ be the set of vertices $x$ such that the number of arcs directed into $x$ that are deleted plus the number of 
uncovered pairs of $M$ incident with $x$ is at least $ \varepsilon n / 4$.  We claim that $|X| < \varepsilon n / 4$.  Indeed, if this
bound does not hold then the number of deleted arcs plus twice the number of uncovered pairs of $M$ is at least
$ (\varepsilon^2/16) n^2 $, which cannot be the case for $k$ sufficiently large by Lemma~\ref{lem:uncover}.  Now consider
the induced digraph $ D'[ [n] \setminus X ] $.  In-degrees in this digraph are at least $ n - k - 1- \varepsilon n / 4 - \varepsilon n /4 $ 
as within $[n] \setminus X$, we have max outdegree $k$ in $M$ and at most $\varepsilon n / 4$ uncovered pairs and deleted in-arcs at a vertex. Thus the minimum in-degree in this induced sub-digraph is at least
\[ n - k - \frac{ n \varepsilon} 2 - 1 = n \left( 1 - \frac{1 - \alpha}{ 1 + \varepsilon} - \frac{ \varepsilon}2 \right)-1  
= n \left(\frac{ \alpha + \varepsilon/2 - \varepsilon^2/2}{ 1 + \varepsilon}\right)-1 > \alpha n, \]
where we use $ \alpha < 2/5$ in the last inequality.  Now, since $\alpha$ is directed triangle sufficient, we
conclude that $ D'[ [n] \setminus X ] $ contains a directed triangle.

Directed Triangle Removal implies that $D$ contains $ \Omega(n^3)$ directed triangles. 
Consider such a triangle $xyz$ and the color $c$ such that $ xyz \in \HH_c$.  No pair among $ xyz$ can be a locked pair for
this color as we have arcs in only one direction.  Furthermore the triple $ xyz$ cannot be contained in one of the stars in $\HH'_c$ as
this would require a vertex of in-degree two in the color $c$ digraph induced on $ xyz$.  We conclude that $xyz$ is a stray triple in
color $c$.  But there are at most $21nk$ stray triples,  and if $k$ is sufficiently large, we can find some directed triangle which 
is not covered by a stray triple and thus is uncolored in our coloring.
This is a contradiction.

\begin{remark}
Applying that $\alpha = .3465$ is directed triangle sufficient \cite{hladky2017}, we have that $r_k(\LL) \leq 1.531k$ for $k$ sufficiently large.
\end{remark}

\section{Messy Path}
The messy path is the hypergraph  $ \M= \{abc,bcd,def\}$. Extremal results for collections of hypergraphs containing $\M$ are studied in \cite{furedi2017}, where the messy path is $P_3(1,2)$ or $P_3(2,1)$, and in \cite{furedi2011}, where the messy path is a $(2,1)$-cluster. 
In \cite{furedi2011}, it is shown that for sufficiently large $n$, $ex(n,\M ) = \binom{n-1}{2}$ with the unique extremal hypergraph being a complete star. In Section~4 we find the extremal number for all $n$.  This bound on the extremal number implies $ r_k(\M) \le 3k$ if $k \geq 3$.

This Section is dedicated to the proof of Theorem~\ref{main:messy}.  The outline of the proof is the same as the for the loose path.  We begin with a structural characterization of $\M$-free hypergraphs.  We then use this characterization to define a colored multidigraph associated
with a $\M$-free $k$-coloring of $ \binom{[n]}{3}$ and proceed to establish our upper bound on $ r_k(\M)$.


\subsection{Messy Path-Free Hypergraph Characterization}
We begin with our characterization of messy path-free hypergraphs.  Note that, like our characterization of $\LL$-free hypergraphs, this 
characterization features disjoint stars.  However, the rest of the characterization is less well-behaved and hence more challenging in the
application that follows.
\begin{theorem}
\label{thm:messychar}
Let $\HH$ be a messy path-free hypergraph, and let $\HH'$ be the 13-core of $\HH$. The vertex set $V(\HH')$ has a partition into 3 sets $X,Y,Z$ and the set $Z$ has a partition into sets $(A_v: v \in Y)$ such that all triples $e$ of the hypergraph $\HH'$ have one of the following two forms:
\begin{itemize}[nosep]
\item $e \subseteq X \cup Y$ where $\HH[X \cup Y]$ is a partial Steiner Triple System
\item $e \cap Y = \{y\}$ and $e \setminus \{y\} \subseteq A_y$
\end{itemize}
\end{theorem}
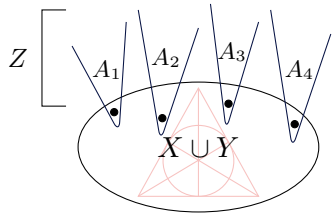
\begin{figure}[ht]
\centering
	\begin{tikzpicture}[scale=1.6]
	\draw[color=re!30] (-.5,-.95) -- (0,-.05) -- (.5,-.95) -- cycle;
	\draw[color=re!30] (-.5,-.95) -- (.25,-.5)
					   (0,-.05)-- (0,-.95)
					   (.5,-.95) -- (-.25,-.5);
	\draw[color=re!30] (.25,-.5) to [closed, curve through = {(0,-.95)}] (-.25,-.5);

	\draw (0,-.54) ellipse (1 and .54) node {$X \cup Y$};

	\filldraw (.25,-.18) circle (.03)
		 	  (-.3,-.3) circle (.03)
			  (-.7,-.25) circle (.03)
			  (.8,-.35) circle (.03);
						
	\draw[rounded corners,color=bl]
				(-1.05,.3) -- (-.66,-.42) -- (-.6,.5)
				(-.5,.6) -- (-.315,-.49) -- (-0,.45)
				(.1,.65) -- (.23,-.4) -- (.5,.5)
				(1.1,.4) -- (.78,-.55) -- (.6,.6);

	\draw (-1.1,.6) -- (-1.3,.6) --(-1.3,-.2) -- (-1.1,-.2);
	\node[font=\small] at (-1.5,.2) {$Z$};
				
	\node[font=\footnotesize] at (-.76,.1) (1) {$A_1$};
	\node[font=\footnotesize] at (-.27,.2) (2) {$A_2$};
	\node[font=\footnotesize] at (.28,.25) (3) {$A_3$};
	\node[font=\footnotesize] at (.84,.1) (4) {$A_4$};
	\end{tikzpicture}
	\caption{$\M$-free hypergraph}
\end{figure}

In the application of Theorem~\ref{thm:messychar} we make use of the following definition.
\begin{definition}
If $ \HH$ is an $\M$-free hypergraph and $ \HH'$ is the 13-core of $\HH$ then the triples in $ \HH \setminus \HH'$ will be called stray triples (or removed triples).
\end{definition}

The remainder of this subsection is a proof of Theorem~\ref{thm:messychar}.  Throughout the proof, we let $\HH$ be an $\M$-free hypergraph and $ \HH'$ be the 13-core of $\HH$. Recall that $F(a,2)$ refers to the 3-uniform hypergraph with vertex set $\{x_1,\dots,x_a,y_1,y_2\}$ and edge set $\{x_iy_1y_2:i\in[a]\}$. We refer to the vertices $x_1,\dots,x_a$ as petals of the $F(a,2)$ and $y_1, y_2$ as the center.

\begin{lemma}
The vertex cover number $\tau(Tr_{\HH'}(v)) \neq 2,3$ for all vertices $v \in V(\HH')$.
\end{lemma}
\begin{proof}
Let $v$ be a fixed vertex. Assume for the sake of contradiction that the vertex cover number of the trace of $v$ in $\HH'$ is 2 or 3. Let $U$ be a minimal vertex cover in $Tr_{\HH'}(v)$. Since $ Tr_{\HH'}(v)$ has at least 13 edges, one of these vertices has degree at least 5, say $u$. Let $x \in U$ with $x \neq u$. 
Note that by definition, there exists an edge $xy$ where $y \neq u$. Then $u$ has a neighbor $a$ in $Tr_{\HH'}(v)$ which is not in $U$ and is neither $x$ nor $y$. Consider a triple $afg$.
\begin{itemize}[nosep]
\item if $u,v \notin afg$ then $\M$ appears. (Consider $afg$ together with the triple corresponding to a path in $ Tr_{\HH'}(v) $ centered at $u$.)
\item if $ v \not\in afg$, $ u \in afg $, and $ \{x,y\} \cap \{ f,g\} = \emptyset $ then $ \M$ appears. (Consider $\{ afg = aug, auv, vxy\}$.)
\end{itemize}
So we may assume that such triples do not appear in $ \HH'$. Now we observe that there are at most 3 triples that contain $a$ and $v$ by the bound on the vertex cover number. There are at most two triples that contain $a$ and $u$ as such triples contain $x$ or $y$. Thus, $a$ is in at most $5$ triples in $\HH'$, which is a contradiction.
\end{proof}
We can now start to identify the parts in the vertex partition set forth in Theorem~\ref{thm:messychar}.  Let $Z$ be the set of vertices whose traces in $\HH'$ have vertex cover number 1.
\begin{lemma}
\label{lem:tau1}
If $z \in Z$ and $y$ is the center of the star in $Tr_{\HH'}(z)$ , then $\tau(Tr_{\HH'}(y)) \geq 4$.  Furthermore, if $u$ is any leaf of the star in $Tr_{\HH'}(z)$, then $\tau(Tr_{\HH'}(u)) = 1$.
\end{lemma}
\begin{proof}
Assume for the sake of contradiction that $\tau(Tr_{\HH'}(y)) = 1$. Then, $y,z$ form the center of a $F(13,2)$ in $\HH'$. Note that any petal of this $F(13,2)$ has degree 1 or else there is a messy path, a contradiction.

Assume for the sake of contradiction that some leaf $u$ has $\tau(Tr_{\HH'}(u)) \geq 4$. Then, we can find a triple in $ \HH'$ that contains $u$ and neither $y$ nor $z$. This triple together with two triples corresponding to a path of two edges in $Tr_{\HH'}(z)$ forms a messy path, a contradiction.
\end{proof}
We now let $Y$ be the set of vertices $y$ for which there is $z \in Z$ such that $y$ is the center of the star in $Tr_{\HH'}(z)$. For $y \in Y$, let $A_y$ be the set of $z \in Z$ for which $y$ is the center of the star in $Tr_{\HH'}(z)$. Finally, let $X = V(\HH') \setminus (Y \cup Z)$.
Note that in $X,Y$, every vertex has a trace with vertex cover number at least 4. Further note that $(A_y : y \in Y)$ partitions $Z$.

\begin{lemma}
\label{lem:pSTS}
The hypergraph $\HH'[X \cup Y]$ is a partial Steiner Triple System.
\end{lemma}
\begin{proof}
Assume for the sake of contradiction that there are triples 
$wxy, xyz \in \HH'[X \cup Y]$. Since $Tr_{\HH'}(w)$ has vertex cover number at least 4, there exists an edge in $Tr_{\HH'}(w)$ not containing $x,y,z$. The triple corresponding to this edge together with $wxy,xyz$ form a messy path. 
\end{proof}
Theorem~\ref{thm:messychar} follows from Lemmas~\ref{lem:tau1}~and~\ref{lem:pSTS}.



\subsection{Multicolor Ramsey number for the Messy Path: Proof of Theorem~\ref{main:messy}}

The proof of Theorem~\ref{main:messy} uses a digraph structure analogous to the structure 
introduced in the proof of Theorem~\ref{main:loose2}.  In that proof it was sufficient to find a cubic number of directed triangles. This is not 
sufficient for the messy path as a cubic number of
triples can be contained in the `Steiner parts' of the color classes.  
So we make a more subtle argument that takes more of the structure into account.


Throughout this section we 
consider a messy path-free $k$-coloring of $\binom{[n]}{3}$ where $k$ is assumed to be large.  We establish upper bounds on $n$ that imply Theorem~\ref{main:messy}. Let $C$ be the set of colors.  For each color $c \in C$, let $ \HH_c $ be the collection of triples colored with color $c$.  We apply Theorem~\ref{thm:messychar}.  For each $ \HH_c$, we let $ \HH'_c $ be the 13-core of $ \HH_c$ and we let $ X_c, Y_c, Z_c$ be the
partition of the vertex set given by Theorem~\ref{thm:messychar}.  Furthermore, for   
each vertex $v$ in the set $Y_c$ let $A_{v,c}$ be the set $A_{v}$ given by Theorem~\ref{thm:messychar}.  
We define a colored multidigraph $M$ on the vertex set $V = [n]$ as follows. The directed arc $(u,v)$ appears in the multidigraph with color $c$ if $u \in A_{v,c}$, i.e., $u$ points to $v$ if $u$ is contained in the body of a star centered at $v$. 
For a specific color, we denote this arc by $(u,v,c)$. 

Define
\[m^-_c(v) = \abs{\{u \in [n]: (u,v,c) \in M\}}, \quad m^-(v) = \sum_{c \in C} m^-_c(v)\]
\[m^+_c(v) = \abs{\{u \in [n]: (v,u,c) \in M\}}, \quad m^+(v) = \sum_{c \in C} m^+_c(v)\]
Note that $m^+(v) \leq k$ as $v$ can be in the body of at most one star in each $ \HH'_c$.

\begin{lemma}
\label{lem:messynoarc}
There are at most $O(k)$ pairs of vertices $\{u,v\}$ such that neither $(u,v)$ nor $(v,u)$ appear as an arc in any color in $M$.
\end{lemma}
\begin{proof}
Note that if neither $(u,v,c)$ nor $(v,u,c)$ appears then there is at most one triple in $ \HH'_c$ that contains the pair $\{u,v\}$.  Therefore $k$ plus the number of stray triples containing $\{u,v\}$ is at least $n-2$, and such a pair $\{u,v\}$ is contained in at least $n-k-2$ stray triples. As there are at most $12nk$ stray triples across the colors, there are at most $\frac{12nk}{n-k-2} = O(k)$ such pairs of vertices.
\end{proof}

We need further notation to extract more refined information from the colored multidigraph $M$.
For each vertex $v$ let $s(v)$ be the number of colors $c$ for which $v \in X_c \cup Y_c $; in other words, $s(v)$ is the number of colors $c$ for which $v$ is in the Steiner triple system for color $c$. Define
\[S = \sum_{v \in [n]} s(v).\]
Note that $v\in A_{u,c}$ implies that $v$ is not part of a partial Steiner triple system in $ \HH'_c$. Hence, we have
\begin{equation}
    \label{eq:zero}
m^+(v) + s(v) \leq k \ \ \ \text{ and }  \ \ \ \abs{ M } = \sum_{v \in [n]} m^+(v) \leq \sum_{v\in [n]} k-s(v) = nk - S.
\end{equation}
We will also want to keep track of stray triples.  For $u,v \in [n]$ define
\begin{align*}
 \xi_{uv} & = \left| \left\{ z : \exists c \in C \text{ such that } uvz \in \HH_c \setminus \HH'_c \right\} \right| \\
\xi_{u} & = \left| \left\{\{y, z\} : \exists c \in C \text{ such that } uyz \in \HH_c \setminus \HH'_c \right\} \right|.
\end{align*}
Note that $\frac{1}{3}\sum_{u \in [n]} \xi_u, \frac{1}{3}\sum_{\{u,v\} \in \binom{[n]}{2}} \xi_{uv} \leq 12nk$. Next, we categorize pairs of vertices.
\begin{itemize}[nosep]
\item We say that a vertex pair $\{u,v\}$ is a two-cycle pair if $(u,v), (v,u) \in M$. Let $t(u)$ be the number of two-cycle pairs that contain $u$. We also define
\[T  = \frac{1}{2} \sum_{u \in [n]} t(u) = \text{ number of two-cycle pairs of vertices}, \qquad \bar{t} = \frac{2T}{n}\]
\item We say a pair $\{u,v\}$ is a parallel pair if $\{u,v\}$ is not a two-cycle pair and there are at least two arcs in $M$ contained in the pair, i.e., $(u,v,c)$ and $(u,v,c')$ appear or $(v,u,c)$ and $(v,u,c')$ appear for $c,c'$ distinct. Let $p^+(u)$ and $ p^-(u)$ be the number of parallel pairs that are directed out of and into $u$, respectively. We also let:
\[P  = \sum_{u \in [n]} p^+(u) =  \text{ number of parallel pairs of vertices}, \qquad \hat{p} = \frac{P}{n-k}\]
\item We say that a pair $\{u,v\}$ is a solo pair if either $(u,v,c)$ or $(v,u,c)$ for some color $c$ is the only arc on the pair. Let $q^+(u)$ and $q^-(u)$ be the number of solo arcs that are directed out of and into $u$ respectively.
\item We say that a pair $\{u,v\}$ is an uncovered pair if neither $(u,v)$ nor $ (v,u)$ appears in $M$.
\end{itemize}
Now, Lemma~\ref{lem:messynoarc} implies that at least $\binom{n}{2} - O(k)$ pairs of vertices are covered with at least one arc, and each two-cycle pairs and each parallel pair requires at least one additional arc. It follows that we have
\begin{equation}
\label{eq:count}
P + T \le \abs{M} - \binom{n}{2} + O(k) \le nk - S - \binom{n}{2} + O(k). 
\end{equation}
Below we establish two upper bounds on $n$, one which is a function of $P$ and another that is a function of $T$.  These bounds taken together with (\ref{eq:count}) imply Theorem~\ref{main:messy}.

We need additional notation. We define an oriented graph $D$ on $[n]$, which defines the one-way out-neighborhood and in-neighborhood $D^+(v),D^-(v)$ (and out-degree and in-degree) for a vertex $v \in [n]$ as follows
\begin{itemize}[nosep]
\item $(v,u) \in D$ if $(v,u) \in M \text{ and } (u,v) \notin M$
\item $D^+(v)=\{u \in [n]: (v,u) \in D\} , d^+(v) = \abs{D^+(v)}$
\item $D^-(v)=\{u \in [n]: (u,v) \in D\}, d^-(v) = \abs{D^-(v)}$
\end{itemize}
Note that $D$ consists of arcs corresponding to solo pairs and parallel pairs. 

Consider a fixed vertex $v$.  Observe that for any other vertex $u$ we either have $ (v,u) \in M$ (and $u$ is counted by $m^+(v)$), $ (u,v) \in M$ and $ (v,u) \not\in M $ (and $u$ is counted by $d^-(v)$), or $\{u,v\}$ is uncovered (and there are at least $n-k$ stray triples that contain $uv$). Therefore $ m^+(v) + d^-(v) + \xi_v/(n-k)  \ge n-1 $. Applying (\ref{eq:zero}) it follows that we have
\begin{equation}
\label{eq:onewayin}
k-s(v)+p^-(v) + q^-(v) \geq m^+(v) + d^-(v) \geq n - \xi_v/(n-k) - 1
\end{equation}

With these preliminary observations in hand, we are now ready to state two key Lemmas. The first Lemma gives a bound on $n$ in terms of $\bar{t}$ while the second Lemma gives a bound on $n$ in terms of $ \hat{p}$. We then complete the proof by combining these Lemmas and the bound on $ P +T$ given in (\ref{eq:count}). 
\begin{lemma}
\label{lem:messytriangles}
If $ \varepsilon>0 $ and $k$ is sufficiently large then
\[n \leq \frac{4k}{3} + \bar{t} + \varepsilon k\]
\end{lemma}
\begin{proof}
First note that 
\[\abs{D} \geq \binom{n}{2} - T - O(k) \]
as every pair of vertices is either a solo pair, a parallel pair, a two-cycle pair, or an uncovered pair. Then note that $\sum_{v \in [n]} d^+(v) + \frac{d^-(v)}{2} = \frac{3}{2}\abs{D}$ and so
\[\E_{v \in [n]} \left[ d^+(v) + \frac{d^-(v)}{2} - \frac{ 2 \xi_v}{ \varepsilon k} \right] \geq \frac{3(n - \bar{t})}{4}  - O(1/\varepsilon). \]
Let $v$ be a vertex that maximizes $ d^+(v) + \frac{d^-(v)}{2} - \frac{  2\xi_v}{\varepsilon k}$.

Note that we may assume $ d^-(v) > \varepsilon k$ as otherwise we have 
\[ k > d^+(v) \ge \frac{3(n - \bar{t})}{4}  - O(1/\varepsilon) -\frac{ \varepsilon k}2 \]
and the desired bound follows.

By Lemma~\ref{lem:messynoarc}, we have $\abs{\{\{u,w\} \in \binom{D^-(v)}{2}: (u,w)\text{ or }(w,u) \in M\}} \geq 
\binom{d^-(v)}{2} - O(k)$.  Therefore
\[ \E_{u \in D^-(v)} \left[  \abs{\left\{ w \in D^-(v) : (u,w) \in M \right\} } - \xi_{uv} \right] \ge \frac{ d^-(v)}{2} - O(1/\varepsilon) - \frac{ 2 \xi_v}{\varepsilon k }. \]
Consider a vertex $ u \in D^-(v)$ that maximizes $\abs{\left\{ w \in D^-(v) : (u,w) \in M \right\}} - \xi_{uv}$. Note that for every vertex $x \in D^+(v)$, the triple $uvx$ appears in some $\HH_c$ in one of the following ways: $uvx$ is contained in the partial Steiner triple system on $X_c \cup Y_c$; $u,v \in A_{x,c}$ so that $(u,x,c)$ and $(v,x,c)$ both appear in $M$; or $uvx$ is one of the stray triples. (Note that we are making of the use of the fact that $D$ only includes arcs that are not in 2-cycles in $M$. In particular $ (x,v)$ and $ (v,u)$ do not appear in $M$.) Therefore, as $D^-(v)$ and $D^+(v)$ are disjoint,
\[ s(u) + m^+(u) -  \abs{\left\{ w \in D^-(v) : (u,w) \in M \right\} }  + \xi_{uv} \ge d^+(v).\]
Recalling $k \geq s(u) + m^+(u)$ from (\ref{eq:zero}) we have
\begin{align*}
k &\geq d^+(v) +  \abs{\left\{ w \in D^-(v) : (u,w) \in M \right\} } - \xi_{uv} \geq 
 d^+(v)   + \frac{ d^-(v)}{2} - O(1/\varepsilon) - \frac{ 2 \xi_v}{\varepsilon k } \\
 &\geq   \frac{3(n - \bar{t})}{4}  - O(1/\varepsilon). 
 \end{align*}
Rearranging and letting $k$ be sufficiently large, we have
\[n \leq \frac{4k}{3} + \bar{t} + \varepsilon k\qedhere\]
\end{proof}

\begin{lemma}
\label{lem:solos}
If $0<\varepsilon<0.01$ and $k$ is sufficiently large then 
\[n \leq \max\left\{1.59k, \frac{3}{2}k + \frac{\hat{p}}{2} + \varepsilon k\right\}\]
\end{lemma}
\begin{proof}
We say that a vertex $v$ is light if $\max_{c \in C} m^-_c(v) \leq n/2$. Note that since each color contributes at most one vertex which is not light, there are at least $n-k$ light vertices.  Let $L$ 
be the collection of light vertices.  By (\ref{eq:onewayin}), we have
\[ \E_{v \in L} \left[ q^-(v) - s(v)  - \frac{6 \xi_v}{ \varepsilon k} \right]  \ge n - k - \hat{p} - O( 1/ \varepsilon). \]
Let $v$ be a vertex of $L$ that maximizes $  q^-(v) - s(v) - \frac{6 \xi_v}{ \varepsilon k} $. 

Observe that, appealing to (\ref{eq:count}),  we have
\[ \hat{p} \leq \frac{ nk - n^2/2 +O(k)}{n-k}.   \]
Assuming that $n \geq 1.59k$ and noting that the above bound is decreasing in $n$ (and consequently maximized when $n = 1.59k$), we have
\begin{equation}
\label{eq:qlin}
q^-(v) - s(v) - \frac{6 \xi_v}{ \varepsilon k} \geq .59k - \frac{1.59k^2-(1.59k)^2/2}{.59k} - O(1/\varepsilon) \geq .0375k - O( 1/ \varepsilon ) 
\geq 2\varepsilon k
\end{equation}
and so $ q^-(v) - s(v)$ is linear in size.

Let $B$ be the set of vertices $u$ such that $ (u,v)$ is a solo arc in $M$. Note that $\abs{B} = q^-(v)$.
We now consider cases based on the colors on the solo arcs directed from $B$ into $v$.  For each color $c$ let $B_c$ be the set of vertices $x \in B$ such
that $xv$ is colored $c$.

\textbf{Case 1.} {\em Some $c$ has $ \varepsilon k < \abs{B_{c}}$.}

Let $ w \in B_c $ such that $ \xi_{wv} \le 2 \xi_{v}/( \varepsilon k) $. Consider the triples of the form $wvz$ where $z \in B \setminus B_{c}$. As $w$ and $z$ point to $v$ with solo arcs of different colors, such a triple must be covered by stray triples or a triple in a partial Steiner triple system. Thus,
\[ s(v) + 2 \xi_{v}/( \varepsilon k)  \ge  s(v) + \xi_{wv} \geq q^-(v) -  \abs{B_{c}}. \] 
(Note that if $B = B_c$, then the above inequality just states $s(v) + 2\xi_{v}/(\varepsilon k) \geq 0$.) Therefore,
\[\abs{B_{c}}  \geq q^{-}(v) - s(v) - 2\xi_{v}/(\varepsilon k)  \geq  n-k - \hat{p} +4\xi_{v}/(\varepsilon k)   - O( 1/\varepsilon).\]
Now note Lemma~\ref{lem:messynoarc} implies that $\abs{\{\{u,z\} \in \binom{B_{c}}{2}: (u,z) \text{ or } (z,u) \in M\}} \geq \binom{\abs{B_c}}{2} - O(k)$ so we have
\[ \E_{u \in B_{c}} \left[  \abs{ \left\{ z \in B_{c} : (u,z) \in M \right\} } - \xi_{uv} \right] \ge \frac{|  B_{c} | }{2} - O(1/\varepsilon) - \frac{ 2 \xi_v}{\varepsilon k }. \]
Consider a vertex $ u \in B_c$ that maximizes 
$ \abs{\left\{ z \in B_c : (u,z) \in M \right\} } - \xi_{uv}$. 
Let $C$ be the set of vertices $y$ such that $ (y,v,c)\not\in M$. Note that for every vertex $y \in C$, the triple $uvy$ appears in some $\HH_{c'}$ in one of the following ways: $uvy$ is contained in the partial Steiner triple system on $X_{c'} \cup Y_{c'}$; $u,v \in A_{y,c'}$ so that $(u,y,c'), (v,y,c') \in M$; or $uvy$ is one of the stray triples. Therefore,
\[s(u) + m^+(u) - \abs{\left\{ z \in B_c : (u,z) \in M \right\}} + \xi_{uv} \geq |C|.\]
Since $k \geq s(u) + m^+(u)$ from (\ref{eq:zero}), we have
\begin{align*}
k &\geq |C| +  \abs{\left\{ z \in B_c : (u,z) \in M \right\} } -\xi_{uv} \geq 
 |C| +   \frac{|  B_c | }{2} - O(1/\varepsilon) - \frac{ 2 \xi_v}{\varepsilon k } \\
 &\geq   \frac{n}{2} + \frac{ n - k - \hat{p} }{2}  - O(1/\varepsilon),
 \end{align*}
where the bound $|C| \ge n/2$ follows from the fact that $v$ is a light vertex.
Rearranging and letting $k$ be sufficiently large, we have
\[n \leq \frac{3k}{2} +  \frac{ \hat{p}}{2}+ \varepsilon k, \]
as desired.

\textbf{Case 2.} {\em  Every color appears on at most $ \varepsilon k$ solo arcs from $B$ to $v$.}

Consider the triples containing $v$ and two vertices in $B$. The solo arcs partition $B$ so that the $\ell$ colors on these arcs color at most $\sum_{i=1}^\ell \binom{x_i}{2}$ of these triples, where $x_i$ is the number of vertices with solo arcs in the $i^{\text{th}}$ color. 
Noting that $x_i \leq \varepsilon k $ and $\sum_{i=1}^\ell x_i = q^-(v)$, we have that by convexity,
\[\sum_{i=1}^\ell \binom{x_i}{2} \leq \sum_{i=1}^\ell \frac{x_i^2}{2} \leq \frac{q^-(v)}{\varepsilon k}\cdot \frac{(\varepsilon k)^2}{2} \leq  \frac{q^-(v)\varepsilon k}{2}\]
Furthermore, the number of triples in the partial Steiner triple systems that contain $v$ and two vertices from $B$ is at most $s(v)\frac{q^-(v)}{2}$, since each partial Steiner triple system contributes a matching in the trace of $v$. The only other triples containing $v$ and two vertices in $B$ are the $\xi_v$ stray triples containing $v$. Then, we have 
\[\binom{q^-(v)}{2} \leq \frac{q^-(v)\varepsilon k}{2} + \frac{s(v)q^-(v)}{2} + \xi_v\]
so $q^-(v)  - s(v) \leq \varepsilon k + 1 + \frac{ 2 \xi_v}{q^-(v)} $ and
\[q^-(v)  - s(v) - \frac{ 6 \xi_v}{ \varepsilon k } 
\leq \varepsilon k + 1.  \]
This contradicts (\ref{eq:qlin}) when $k$ is sufficiently large.
\end{proof}

We are now ready to complete the proof.  Lemmas~\ref{lem:messytriangles} and \ref{lem:solos} imply that if either $ \hat{p}$ or $ \bar{t}$ is sufficiently small then the desired upper bound on $n$ follows. 
Observe that (\ref{eq:count}) implies
\begin{equation}
    \label{eq:stepper}
\hat{p} (n-k) + \frac{n \bar{t}}{2} \le nk - \frac{n^2}{2} + O(k).
\end{equation}
Given this linear relationship, we see that $n$ is at most the bound determined by setting the expressions in the two Lemmas equal. So, we set
\[\frac{4k}{3} + \bar{t} + \varepsilon k = \frac{3k}{2} + \frac{\hat{p}}{2} + \varepsilon k \quad \Rightarrow \quad \bar{t} = \frac{k}{6} + \frac{\hat{p}}{2}\]
Applying this to (\ref{eq:stepper}), we have
\[\hat{p}(n-k) + \frac{n(\frac{k}{6} + \frac{\hat{p}}{2})}{2} \leq nk - \frac{n^2}{2} + O(k) \quad \Rightarrow\quad \hat{p}\left(\frac{5n}{4}-k\right) \leq \frac{11}{12}nk - \frac{n^2}{2} + O(k)\]
Therefore, $n$ is at most the maximum of $1.59k$ and
\begin{equation}
    \label{step2}
\frac{3k}{2} + \frac{11nk-6n^2+ O(k)}{2(15n-12k)} + \varepsilon k
\le 
\frac{3k}{2} + \frac{11nk-6n^2}{2(15n-12k)} + 2\varepsilon k.
\end{equation}
Note that the right hand side of (\ref{step2}) is decreasing in $n$.
So, assuming $ n \ge 1.59k$, we conclude $ n \le \eta k + 2\varepsilon k$, where $\eta$ is a root of the equation
\[(2\eta-3)(15\eta-12) = 11\eta - 6\eta^2 \]
in the interval $[1,3]$. As this quadratic equation simplifies to
$9\eta^2 - 20 \eta + 9 = 0$, we have
$$n \leq \left(\frac{10+\sqrt{19}}{9}+2\varepsilon\right)k$$ for sufficiently large $k$, as desired.


\section{Extremal number of the messy path}
We define $S^{(3)}_{n-1}$ to be the complete 3-uniform star on $n$ vertices; that is,
$S^{(3)}_{n-1}$ is the collection of all triples on $n$ vertices that contain some fixed vertex.
F\"uredi and \"Ozkahya proved that for sufficiently large $n$, $ex(n,\M) = \binom{n-1}{2}$ and $ S^{(3)}_{n-1}$ is the unique extremal hypergraph \cite{furedi2011}. Here we observe that this result extends to smaller $n$.

Of course, the star is not the unique extremal hypergraph in $ Ex(n, \M)$ when $n$ is too small. For example, the complete 3-uniform hypergraph on $n$ vertices does not contain $ \M$ when $ n \le 5$. Furthermore, if $X$ is a vertex set of size $n=6$ then any collection $ {\mathcal F}$ of 3-sets with the property that $ e \in {\mathcal F} $ implies $ X \setminus e \not\in {\mathcal F}$ does not contain a copy of the messy path.  The collection of all such hypergraphs is given explicitly in \cite{pol2017int}. Our main result here is that the star is the unique extremal hypergraph in $ Ex(n, \M)$ if $n \ge 7$.

In the proof we make use of hypergraph $ F(a,2)$ which was introduced above. Recall that $F(a,2)$ refers to the 3-uniform hypergraph with vertex set $\{x_1,\dots,x_a,y_1,y_2\}$ and edge set $\{x_iy_1y_2:i\in[a]\}$. We refer to the vertices $x_1,\dots,x_a$ as petals of the $F(a,2)$ and $y_1, y_2$ as the center.

\begin{remark}
A family $\mathcal{F} \subset \binom{X}{3}$ that does not contain a messy path has the property that if $A,B,C \in \mathcal{F}$ are distinct sets such that  $\abs{A \cup B \cup C} =6 $ then $A \cap B \neq \varnothing, A \cap C \neq \varnothing, B \cap C \neq \varnothing$. It is known that for $n$ sufficiently large, the size of such a family is at most $\binom{n-1}{2}$, with the unique extremal family being a star \cite{furedi2011}. This extremal result is also known for other similar conditions on families of sets generalizing the Erd\H{o}s-Ko-Rado Theorem \cite{erdos1961, frankl1983, mubayi2006, mubayi2005}
\end{remark}

\begin{theorem}
\[ex(n, \M ) = \begin{cases}\binom{n}{3}&\text{if }n\leq 5\\\binom{n-1}{2}&\text{if }n \geq 6.\end{cases}\]
Furthermore, if $n \ge 7$ then $S^{(3)}_{n-1}$ is the unique hypergraph in
$Ex(n, \M)$. 
\end{theorem}
\begin{proof}
For $n \leq 5$, since a messy path has 6 vertices, $\binom{[n]}{3}$ is the unique extremal hypergraph. For $n=6$, note that if there is a pair of disjoint triples, then any other edge would create a messy path. So we may assume the hypergraph is intersecting, and Proposition 1.6 in \cite{pol2017int} gives the desired result.

For $n \geq 7$, we proceed by induction. For the base case of $n = 7$, it suffices to show that the family is intersecting---then the result follows by Erd\H{o}s-Ko-Rado bound. So, suppose there are two disjoint triples $e,f$. Note that any other triple contained in their union would create a messy path. Thus, all other triples must contain the last vertex. Note that if this vertex is incident with an triple intersecting $e$ in two vertices and another triple intersecting $f$ in two vertices, then this would also create a messy path. Thus, there are no triples contained among this vertex and for instance, $f$. This implies that there are at most $\binom{6}{2}-3+2$ triples. Thus, the star is the unique extremal family by Erd\H{o}s-Ko-Rado.

For $n \geq 8$, we assume that we have messy path-free family $\HH$ on vertex set $V$ with at least $\binom{n-1}{2}$ triples. Note that if some vertex has degree at most $n-2 = \binom{n-2}{1}$, then the remaining $n-1$ vertices are messy path-free with at least $\binom{n-2}{2}$ triples, so by the inductive hypothesis, this is exactly $\binom{n-2}{2}$ triples in a complete star. Note that if any triple from the removed vertex does not contain the star center then we have a messy path. Thus, the extremal family would be a complete star having exactly $\binom{n-1}{2}$ triples.

So, it suffices to show that $\HH$ has a vertex of degree at most $n-2$. Assume for the sake of contradiction that no such vertex exists.

Note that the average codegree in $ \HH$ satisfies
\[ \frac{ 3 \abs{\HH} }{ \binom{n}{2}} \ge \frac{ 3\binom{n-1}{2}}{ \binom{n}{2}} > 2. \]
So, we may assume $\HH$ has a pair of vertices $u,v$ with co-degree at least 3.

We first consider the case where there exists two vertices $u,v$ of codegree at least 4. Consider any two petals $x,y$ of the resulting $F(a,2)$ with center $\{u,v\}$. Recall that, by assumption, both $x$ and $y$ have degree at least $n-1$. Note that every triple through $x$ or $y$ must intersect $u$ or $v$ or we have a messy path. Thus, each of $x,y$ has at least $n-2$ triples intersecting exactly one of $u$ or $v$, of which at least $n-4$ do not contain both $x,y$. Since $n-4 \geq 3$, we note that $x$ must have at least 2 triples intersecting say, without loss of generality, $u$ which do not contain $v$. Let these triples be $uxa, uxb$. Then, if $y$ is in a triple $vyc$ for $c$ not $u,x$, taking $vyc$, one of $uxa$ and $uxb$ (as to not include $c$), and $uvy$, we would have a messy path. Thus, all triples through $y$ except possibly $vyx$ must intersect $u$. Let another petal of this $F(a,2)$ be $z$. Repeat this argument with $x,z$ to get that $z$ has at least 2 triples through $u$ but not $v,x,y$. Finally, repeating this again with $z,y$ to get that all triples through $y$ must intersect $u$ except possibly $vyz$, we conclude that all triples through $y$ must contain $u$, so $y$ has degree at most $n-2$, a contradiction.

Otherwise all pairs of vertices have codegree at most 3. Now consider any petal $x$ of the $F(3,2)$ with center $\{u,v\}$ and recall that it has degree at least $n-1$. Note that at most one triple through $x$ can avoid $u,v$ in this case. Thus, $x$ has at least $n-3$ triples intersecting exactly one of $u$ or $v$. Then, we have that $\max(deg(x,u), deg(x,v)) \geq 1 + \frac{n-3}{2}$ where the 1 counts $uvx$. Then, $\frac{n-1}{2} \geq \frac{7}{2} > 3$, a contradiction.

Thus, there must exist a vertex of degree at most $n-2$ and by induction, the proof is complete.
\end{proof}


\section{Conclusion}

We emphasize that we spent little effort optimizing for the second order terms (i.e., $\varepsilon$) in Theorems 2 and 3. One reason for this is that we suspect that incremental improvements on the bounds we prove here can be achieved with a bit more effort. In other words, we suspect that our results here do not give the correct asymptotics of these multicolor Ramsey numbers.  The main barrier we see to significant improvement on our upper bounds is the proliferation of 2-cycles in the digraphs introduced in Sections 2.2 and 3.2 that can occur when $n$ is close to $k$.  Indeed, all of the methods that we use in this work are based on pairs of vertices with arcs oriented in only one direction in these digraphs.  It would be interesting to see methods that could handle these 2-cycles and thereby produce upper bounds on $ r_k(\HH)$ that are dramatically closer to $k$.

There are a number of other 3-edge triple systems for which the multicolor Ramsey number is an interesting open question.  These include the loose cycle $ {\mathcal C} = \{ abc , cde, afe \}$ and the hypergraph $ {\mathcal F}_5 = \{ abc, abd, cde \}$. The best known bounds for these multicolor Ramsey numbers are as follows. For the loose cycle, 
\[k+5 \leq r_k(\mathcal{C}) \leq 3k\text{ for }k \geq 3,\]
analogous to the simple bounds for $\mathcal{L}$ and $\mathcal{M}$ \cite{gyarfas2012loose}. For $\mathcal{F}_5$, 
\[2^{ck} \leq r_k(F_5) \leq k!\text{ for }k \geq 4\text{ and }c\text{ some positive constant,}\]
which resembles bounds for $K^{(3)}_4-e$ and $K_3$ \cite{axen2014}. The hypergraph $ {\mathcal G} = \{ abc, abd, bef \}$, which we dub the giraffe, was addressed in the masters thesis of the second author, who showed 
\[k+1 \leq r_k(\mathcal{K}) \leq r_k(\mathcal{G}) \leq k+4,\]
where the kite $\mathcal{K} = \{abc, abd\}$ has $k+1 \leq r_k(\mathcal{K}) \leq k+3$ \cite{axen2014}.

\bibliographystyle{plain}
\bibliography{pathsbib}
\end{document}